\documentclass{amsart}

\usepackage{amsmath, amsthm, amssymb}
\usepackage[dvips]{graphics}

\usepackage{graphicx}
\newtheorem{theorem}{Theorem}[section]
\newtheorem{lemma}[theorem]{Lemma}

\newtheorem{corollary}[theorem]{Corollary}

\theoremstyle{definition}

\theoremstyle{remark}
\newtheorem{remark}[theorem]{Remark}

\numberwithin{equation}{section}

    \usepackage{color}  
     \definecolor{red}{rgb}{0.9,0,0}
     \definecolor{green}{rgb}{0,0.6,0}
     \definecolor{rb}{rgb}{0.6,0,0.2}     
     \definecolor{pass}{rgb}{0,0,0.7}
\usepackage{tikz}

\newcommand{\pt}{\partial}

\newcommand{\R}{\mathbb{R}}
\newcommand{\N}{\mathbb{N}}

\newcommand{\LL}{{\mathcal L}}

\newcommand{\MM}{{\mathcal M}}

\newcommand {\beq} {\begin{equation}}
\newcommand {\eeq} {\end{equation}}
\newcommand {\beqa} {\begin{eqnarray}}
\newcommand {\eeqa} {\end{eqnarray}}
\newcommand {\beqann} {\begin{eqnarray*}}
\newcommand {\eeqann} {\end{eqnarray*}}

\begin{document}

\title[Logarithm in maximum norm error bounds for 3d linear elements]%
{
Logarithm cannot be removed in maximum norm\\ error estimates for linear finite elements in 3D}

\author{Natalia Kopteva}
\address{Department of Mathematics and Statistics,
University of Limerick, Limerick, Ireland}
\email{natalia.kopteva@ul.ie}


%

\date{}


\begin{abstract}
For linear finite element discretizations of the Laplace equation in three dimensions,
we give an example of a tetrahedral mesh in the cubic domain for which the logarithmic factor cannot be
removed from the standard 
upper bounds on the error in the maximum norm.
\end{abstract}

\maketitle

\section{Introduction}
Consider
the problem
\beq\label{prob}
-\triangle u=f\quad\mbox{in\;\;}\Omega\subset\R^d,\qquad u=0 \quad\mbox{on\;\;}\pt\Omega,
\eeq
where $d\in\{2,3\}$, and
$\triangle$ is the  Laplace operator
(defined by $\triangle=\pt_x^2+\pt_y^2+\pt_z^2$ for $d=3$).
It is well known that
a standard linear finite element approximation $u_h$ for~\eqref{prob}
on a quasi-uniform mesh of diameter $h$
is quasi-optimal in the sense that
\beq\label{Ritz_er_Linf}
\|u-u_h\|_{L_\infty(\Omega)}\lesssim h^2|\ln h|\,\| u\|_{W^2_\infty(\Omega)},
\eeq
see, e.g.,
\cite[Theorem~2]{Sch_80},
\cite[Theorem~3.1]{Leyk_Vexler_sinum16} and \cite[Theorem~5.1]{SchWa_best},
for, respectively,
polygonal, convex polyhedral and smooth domains.
It was also shown in \cite{haverkamp,BLRan86,Andreev89_proc}  that when $d=2$, the logarithmic factor cannot be removed from \eqref{prob}.
Surprisingly, it appears that this is still an open question for $d=3$.
In this note, we address  this by giving an example of a tetrahedral  mesh
in the cubic domain $\Omega=(0,1)^3$
such that, under certain conditions on $f$, one has $\|u-u_h\|_{L_\infty(\Omega)}\ge C^*h^2|\ln h|$,
where a positive constant $C^*$ depends on $f$, but not on $h$.

Our mesh is constructed using the two-dimensional triangulation used in a less known paper~\cite{Andreev89_proc}.
Furthermore, the general approach here also follows \cite{Andreev89_proc} in that, on a particular mesh,
we estimate the error of a finite element method
using its explicit finite difference representation.

The paper is organized as follows. In \S\ref{sec_1} we describe the mesh used in our example,
give a finite difference representation of the resulting finite element method, and state the main results. Most of the proofs are deferred to \S\ref{sec_2}.
\smallskip

{\it Notation.}
We write
$a \lesssim b$ when $a \le Cb$ with a generic constant $C$ depending on $\Omega$ 
and $f$,
but not on $h$.
For any domain $\mathcal{D}$ in $\R^2$ or $\R^3$,
the notation $C^{k,\alpha}(\bar{\mathcal{D}})$
is used for the standard H\"{o}lder space
consisting of functions whose $k$th order partial derivatives are uniformly H\"{o}lder continuous with exponent $\alpha\in(0,1)$.
We also use the standard space $L_\infty(\mathcal{D})$ and the related  Sobolev space $W_\infty^2(\mathcal{D})$.


\section{Mesh description. Main results}\label{sec_1}
Consider the standard linear finite element method with lumped-mass quadrature
(for the case without quadrature, see Corollary~\ref{cor}).
With the standard finite element space $S_h$ of continuous piecewise-linear functions
vanishing on $\pt\Omega$,
the computed solution $u_h\in S_h$
is required to satisfy
\beq\label{fem}
\int_\Omega \nabla u_h\cdot \nabla v_h
=\int_{\Omega} (fv_h)^I
\qquad\forall v_h\in S_h,
\eeq
where
$v^I\in S_h$, for any $v\in C(\bar\Omega)$, denotes its  standard piecewise-linear Lagrange interpolant.

Let $\Omega:=(0,1)^3$ and define a tetrahedral mesh in $\Omega$ as follows.
With an even integer $N$, set $h:=N^{-1}$.
Starting with the uniform rectangular grid $\{(x_i,y_j)=(ih,jh)\}_{i,j=0}^N$, 
define a triangulation of $(0,1)^2$
by drawing diagonals as on Fig.\,\ref{fig1} (left) \cite{Andreev89_proc}.
Note that each interior node is shared by 6 triangles, except for $(\frac12,\frac12)$, which is shared by 4 triangles.

Next, partition  $\Omega$ into triangular prisms by constructing a tensor product of the two-dimensional triangulation in the
$(x,y)$-plane and the uniform grid $\{z_k=kh\}$ in the $z$-direction.
Finally, divide each triangular prism into three tetrahedra as on Fig.\,\ref{fig1} (centre, right) using method~A or method~B.
Note that for the resulting tetrahedral mesh to be well-defined, no prisms of the same type can share a vertical
face.
Such a well-defined mesh is generated if, for example,
the shadowed triangles on Fig.\,\ref{fig1} (left) correspond to method B of prism partition.
Note also that it will be convenient to
evaluate the local contributions to \eqref{fem} associated with triangular prisms,
the set of which (rather than of all tetrahedra)
will be denoted  $\mathcal T$.

Introduce the notation
$U_{ijk}:=u_h(x_i,y_j,z_k)$ for nodal values of $u_h$ and the standard finite difference operators defined, for $t=x,y,z$, by
$$
\delta_t^2 U_{ijk}:=
\frac{u_h(P_{ijk}+h\mathbf{i}_t)-2u_h(P_{ijk})+u_h(P_{ijk}-h\mathbf{i}_t)}{h^2},\qquad P_{ijk}:=(x_i,y_j,z_k),
$$
where $\mathbf{i}_t$ is the unit vector in the $t$-direction.
Now, we claim that the finite element method~\eqref{fem} can be rewritten, for $i,j,k=1,\ldots, N-1$, as
\begin{subequations}\label{fdm_full}
\begin{align}\label{fdm}
\LL^h U_{ijk}&:=-(\delta^2_x +\delta^2_y +\gamma_{ij}\delta^2_z) U_{ijk}
=
\gamma_{ij}\,f(x_i,y_j,z_k),
\\[3pt]\label{gamma}
&\mbox{where}\quad
\gamma_{ij}:=\left\{\begin{array}{cl}
\frac23,& \mbox{if}\;i=j=\frac{N}2,\\
1,& \mbox{otherwise},
\end{array}\right.
\end{align}
\end{subequations}
subject to $U_{ijk}=0$ for any $(x_i,y_j,z_k)\in\pt\Omega$.

Indeed, for a particular prism $T\in\mathcal T$, using the notation $E_{lm}$ for the edge connecting vertices $l$ and $m$
(see Fig.\,\ref{fig1}),
and assuming that $E_{23}$ is parallel to the $x$-axis,
a calculation shows that
\begin{align}\label{local}
\int_{T}\nabla u_h\cdot &\nabla v_h
=
\textstyle\frac13 |T|\Bigl\{
2(\pt_x u_h \,\pt_x v_h)\bigr|_{E_{23}}
+
(\pt_x u_h \,\pt_x v_h)\bigr|_{E_{2'3'}}
\\\notag
&{}+ (\pt_y u_h \,\pt_y v_h)\bigr|_{E_{12}}+ 2(\pt_y u_h \,\pt_y v_h)\bigr|_{E_{1'2'}}
+\sum_{l=1}^3(\pt_z u_h \,\pt_z v_h)\bigr|_{E_{ll'}}
\Bigr\}.
\end{align}
Here
we used the fact that each tetrahedron has
an edge parallel to each of the coordinate axes.
Also, within the prism $T$, each of such edges belongs to exactly 1 tetrahedron,
except for $E_{23}$ and $E_{1'2'}$, while each of the latter is shared  by 2 tetrahedra.

Next, set $v_h:=\phi_{ijk}$ in \eqref{fem}, where $\phi_{ijk}\in S_h$ equals $1$ at $(x_i,y_j,z_k)$ and vanishes at all other mesh nodes.
With this $v_h$, adding the contributions of~\eqref{local} to the left-hand side of \eqref{fem}, one gets
$$
\textstyle\frac13 |T|\Bigl\{
-6\delta^2_x U_{ijk}-6\delta^2_y U_{ijk}-6\gamma_{ij}\delta^2_z U_{ijk}\Bigr\}
=\textstyle\frac1{12} |T|\Bigl\{
24\gamma_{ij}\,f(x_i,y_j,z_k)\Bigr\}.
$$
For the right-hand side here, we used the observation that each node $(x_i,y_j,z_k)$ is shared by
$24\gamma_{ij}$ tetrahedra.
Clearly, the above relation immediately implies \eqref{fdm}.

Now that our finite element method \eqref{fem} is represented as a finite difference scheme \eqref{fdm_full},
note that if $\gamma_{ij}$ were equal to $1$ for all $i,j$, we would immediately get the standard finite-difference
error bound $\|u-u_h\|_{L_\infty(\Omega)}\lesssim h^2$ \cite{Sam}. However, $\gamma_{\frac{N}2,\frac{N}2}=\frac23\neq 1$
results in a slightly worse convergence rate, consistent with \eqref{Ritz_er_Linf}.

\begin{theorem}\label{theo_main}
\label{theo_main}
Let $\Omega=(0,1)^3$ and
$f:= F(x,y)\,\sin(\pi z)$ in \eqref{prob}, with any
$F\in C^{2,\alpha}([0,1]^2)$ subject to $F=0$ at the corners of $(0,1)^2$, and
 $F(\frac12,\frac12)=\|F\|_\infty>0$
 (where $\|\cdot\|_\infty$ is used for the norm in $L_\infty((0,1)^2)$). Then
 $u\in C^{2,\alpha}(\bar\Omega)\subset W^2_\infty(\Omega)$, and
 there exists
 a positive constant $C^*$ depending on $F$, but not on $h$,
 such that for
 the finite element approximation $u_h$ obtained
 using \eqref{fem} on the above tetrahedral mesh with a sufficiently small $h$, one has
 $\|u-u_h\|_{L_\infty(\Omega)}\ge C^*h^2|\ln h|$.
\end{theorem}

\begin{corollary}\label{cor}
The result of Theorem~\ref{theo_main} remains valid for a version of \eqref{fem} without quadrature.
\end{corollary}

\begin{proof}
Let $\bar u_h$ be the finite element solution obtained using linear finite elements without quadrature.
Then $|u_h-\bar u_h|\lesssim h^2|\ln h|^{2/3}\|f\|_{W^2_\infty(\Omega)}$; see \cite[final 3 lines in Appendix~A]{kopteva_frac2d}.
The desired assertion follows.
\end{proof}%

\begin{remark}[Additional solution smoothness does not
improve the accuracy]
Under the conditions of Theorem~\ref{theo_main}, for any $m\in\N$, one can choose
$F\in C^{2m,\alpha}([0,1]^2)$ subject to $F=0$ in small neighbourhoods of the corners of $(0,1)^2$.
Then  the theorem remains valid, while now
$u\in C^{2m+2,\alpha}(\bar\Omega)$
(as an inspection of the proof of this theorem reveals; see, in particular, Lemma~\ref{lem1} and Remark~\ref{rem_volkov}).
Thus, additional smoothness of the exact solution would not improve
the accuracy of the finite element method.
\end{remark}

\begin{figure}[t!]
\centering
~\hfill
\begin{tikzpicture}[scale=0.25]

\path[draw, fill=gray!30] (0,0)--(0,2)--(2,0) --cycle;
\path[draw, fill=gray!30] (2,0)--(2,2)--(4,0) --cycle;
\path[draw, fill=gray!30] (4,0)--(4,2)--(6,0) --cycle;
\path[draw, fill=gray!30] (0,2)--(0,4)--(2,2) --cycle;
\path[draw, fill=gray!30] (2,2)--(2,4)--(4,2) --cycle;
\path[draw, fill=gray!30] (4,2)--(4,4)--(6,2) --cycle;
\path[draw, fill=gray!30] (0,4)--(0,6)--(2,4) --cycle;
\path[draw, fill=gray!30] (2,4)--(2,6)--(4,4) --cycle;
\path[draw, fill=gray!30] (4,4)--(4,6)--(6,4) --cycle;

\path[draw, fill=gray!30] (6,6)--(8,6)--(6,4) --cycle;
\path[draw, fill=gray!30] (8,6)--(10,6)--(8,4) --cycle;
\path[draw, fill=gray!30] (10,6)--(12,6)--(10,4) --cycle;
\path[draw, fill=gray!30] (6,4)--(8,4)--(6,2) --cycle;
\path[draw, fill=gray!30] (8,4)--(10,4)--(8,2) --cycle;
\path[draw, fill=gray!30] (10,4)--(12,4)--(10,2) --cycle;
\path[draw, fill=gray!30] (6,2)--(8,2)--(6,0) --cycle;
\path[draw, fill=gray!30] (8,2)--(10,2)--(8,0) --cycle;
\path[draw, fill=gray!30] (10,2)--(12,2)--(10,0) --cycle;

\path[draw, fill=gray!30] (6,6)--(6,8)--(4,6) --cycle;
\path[draw, fill=gray!30] (4,6)--(4,8)--(2,6) --cycle;
\path[draw, fill=gray!30] (2,6)--(2,8)--(0,6) --cycle;
\path[draw, fill=gray!30] (6,8)--(6,10)--(4,8) --cycle;
\path[draw, fill=gray!30] (4,8)--(4,10)--(2,8) --cycle;
\path[draw, fill=gray!30] (2,8)--(2,10)--(0,8) --cycle;
\path[draw, fill=gray!30] (6,10)--(6,12)--(4,10) --cycle;
\path[draw, fill=gray!30] (4,10)--(4,12)--(2,10) --cycle;
\path[draw, fill=gray!30] (2,10)--(2,12)--(0,10) --cycle;

\path[draw, fill=gray!30] (12,12)--(12,10)--(10,12) --cycle;
\path[draw, fill=gray!30] (10,12)--(10,10)--(8,12) --cycle;
\path[draw, fill=gray!30] (8,12)--(8,10)--(6,12) --cycle;
\path[draw, fill=gray!30] (12,10)--(12,8)--(10,10) --cycle;
\path[draw, fill=gray!30] (10,10)--(10,8)--(8,10) --cycle;
\path[draw, fill=gray!30] (8,10)--(8,8)--(6,10) --cycle;
\path[draw, fill=gray!30] (12,8)--(12,6)--(10,8) --cycle;
\path[draw, fill=gray!30] (10,8)--(10,6)--(8,8) --cycle;
\path[draw, fill=gray!30] (8,8)--(8,6)--(6,8) --cycle;

\path[draw,->,help lines] (-0.5,0) -- (14,0) node[below] {\color{black}$x$};
\draw[->,help lines] (0,-0.5) -- (0,14) node[left] {\color{black}$y$};

\path[draw]  (0,0)node[below left]{$0$}--(0,12)node[left]{$1$}--(12,12)--(12,0)node[below]{$1$}--cycle;

\path[draw]  (6,0)node[below ]{$\frac12$}--(12,6)--(6,12)--(0,6)node[left]{$\frac12$}--cycle;
\path[draw]  (6,4)--(8,6)--(6,8)--(4,6)--cycle;
\path[draw]  (6,2)--(10,6)--(6,10)--(2,6)--cycle;

\path[draw]  (0,4)--(4,0);
\path[draw]  (0,2)--(2,0);
\path[draw]  (12,8)--(8,12);
\path[draw]  (12,10)--(10,12);
\path[draw]  (8,0)--(12,4);
\path[draw]  (10,0)--(12,2);
\path[draw]  (0,8)--(4,12);
\path[draw]  (0,10)--(2,12);

\path[draw]  (2,0)--(2,12);
\path[draw]  (4,0)--(4,12);
\path[draw]  (6,0)--(6,12);
\path[draw]  (8,0)--(8,12);
\path[draw]  (10,0)--(10,12);

\path[draw]  (0,2)--(12,2);
\path[draw]  (0,4)--(12,4);
\path[draw]  (0,6)--(12,6);
\path[draw]  (0,8)--(12,8);
\path[draw]  (0,10)--(12,10);

\end{tikzpicture}
\hfill
\begin{tikzpicture}[scale=0.22]

\draw[->,help lines] (3,3) -- (-2.5,-2.5) node[right] {\color{black}$x$};
\draw[->,help lines] (3,3) -- (15,3) node[below] {\color{black}$y$};
\draw[->,help lines] (3,3) -- (3,15) node[left] {\color{black}$z$};

\path[draw]  (0,0)node[draw,shape=circle,left,inner sep=1pt] {\small$3'$}--(12,3)node[draw,shape=circle,below,inner sep=1pt] {\small$1'$}--
(12,12)node[draw,shape=circle,right,inner sep=2pt] {\small$1$}--
(0,9)node[draw,shape=circle,left,inner sep=2pt] {\small$3$}--cycle;
\path[draw, dashed]  (3,12)--(3,3)--(12,3);
\path[draw]  (12,3)--(12,12)--(3,12);
\path[draw, dashed]  (0,0)--(3,3);
\node[draw, shape=circle,above right,inner sep=1pt] at (3,3) {\small$2'$};
\node[draw, shape=circle,above right,inner sep=2pt] at (3,12) {\small$2$};
\path[draw]  (0,9)--(3,12);

\path[draw] (0,9)--(12,3);
\path[draw, dashed] (12,3)--(3,12);
\path[draw, dashed] (0,9)--(3,3);

\end{tikzpicture}
\hfill
\begin{tikzpicture}[scale=0.22]
\path[draw, fill=gray!30] (0,9)--(9,9)--(12,12) --cycle;

\path[draw]  (0,0)--(9,0)--(9,9)--(0,9)--cycle;
\path[draw]  (0,0)--(0,9)--(12,12)--(12,3);
\path[draw, dashed]  (0,0)--(12,3);
\path[draw]  (9,9)--(12,12);
\path[draw]  (9,0)--(12,3);

\path[draw] (9,0)--(0,9);
\path[draw] (12,3)--(9,9);
\path[draw, dashed] (0,9)--(12,3);

\node[draw, shape=circle,left,inner sep=1pt] at (0,9) {\small$1'$};
\node[draw, shape=circle,left,inner sep=2pt] at (0,0) {\small$1$};
\node[draw, shape=circle,right,inner sep=1pt] at (12,12) {\small$3'$};
\node[draw, shape=circle,right,inner sep=2pt] at (12,3) {\small$3$};
\node[draw, shape=circle,right,inner sep=1pt] at (9,9) {\small$2'$};
\node[draw, shape=circle,below,inner sep=2pt] at (9,0) {\small$2$};
\end{tikzpicture}
\hfill~\vspace{-0.2cm}
\caption{Two-dimensional triangulation in $(0,1)^2$ for $N=6$ (left);
partition of a triangular prism into 3 tetrahedra using method A (centre) and method B (right).}
\label{fig1}
\end{figure}
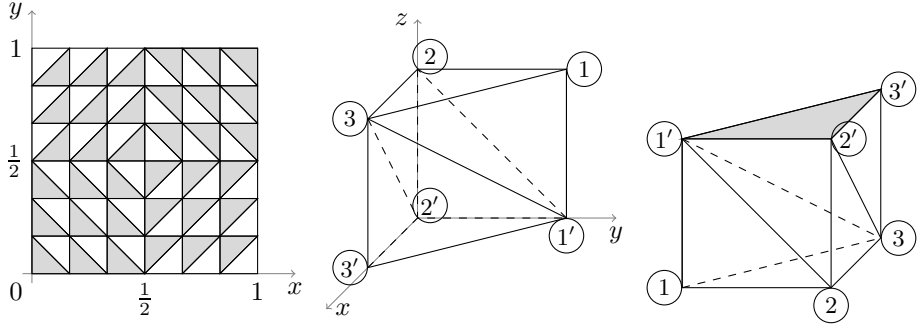

\section{Proof of Theorem~\ref{theo_main}}\label{sec_2}

We split the proof into a number of lemmas, which involve an auxiliary function $w(x,y)$, as well as its finite difference approximation $W_{ij}$, defined by
\begin{align}\label{w_prob}
\MM w&:=-(\pt_x^2+\pt_y^2)w+\pi^2 w=F\quad\mbox{in\;\;}(0,1)^2,
\\[2pt]\label{w_fdm}
\MM^h W_{ij}&:=-(\delta_x^2+\delta_y^2)W_{ij}+\gamma_{ij}\,\pi^2 W_{ij}=\gamma_{ij}\,F(x_i,y_j),
\end{align}
the latter for $i,j=1,\ldots,N-1$,
subject to $w=0$ and $W_{ij}=0$ on the boundary of~$(0,1)^2$.

\begin{lemma}\label{lem1}
Under the conditions of Theorem~\ref{theo_main} on $f$,
the solution of problem~\eqref{prob} is $u=w(x,y)\sin(\pi z)$, where $w\in C^{2,\alpha}([0,1]^2)$ is a unique solution
of \eqref{w_prob} 
and
$|\pt_x^4w|+|\pt_y^4w|\lesssim 1$ in $[0,1]^2$.
Furthermore,
\beq\label{tilde_F_bound}
\textstyle
\widetilde F(\frac12,\frac12)\ge \|F\|_\infty/{\cosh(\frac12\pi)},\qquad
\mbox{where}\;\; \widetilde F:=F-\pi^2 w.
\eeq
\end{lemma}

\begin{proof}
The regularity of $w$ and the bounds on its pure fourth partial derivatives follow from
\cite[Remark~4 in~\S8]{volkov}.
For \eqref{tilde_F_bound}, in view of $F(\frac12,\frac12)=\|F\|_\infty$, it suffices to show that
$$
\|F\|_\infty-\pi^2 w(x,y)\ge B(x):=\|F\|_\infty\frac{\cosh(\pi(x-\frac12))}{\cosh(\frac12\pi)}\,.
$$
The latter is obtained by an application of the maximum principle for the operator $\MM$
as $\MM(\|F\|_\infty-\pi^2 w)\ge 0 = \MM B$, while $\|F\|_\infty-\pi^2 w=\|F\|_\infty\ge B$ on $\pt\Omega$.
\end{proof}

\begin{remark}\label{rem_volkov} In Lemma~\ref{lem1}, we have $w\in C^{2,\alpha}([0,1]^2)$ rather than $w\in C^{4,\alpha}([0,1]^2)$,
as  the latter requires additional corner compatibility conditions on $F$ \cite[Theorem~3.1]{volkov},
while bounded fourth pure partial derivatives of $w$ are sufficient for the finite-difference-flavoured  analysis that yields the crucial bound \eqref{iterm} below.
\end{remark}

\begin{lemma}\label{lem2}
For $U_{ijk}$ of \eqref{fdm_full} and $W_{ij}$ of \eqref{w_fdm} one has
$|U_{ijk}- W_{ij}\sin(\pi z_k)|\lesssim h^2$.
\end{lemma}

\begin{proof}
First, note
that
$-\delta_z^2 [\sin(\pi z_k)]=\lambda_h \sin(\pi z_k)$
where $\lambda_h:=\frac{4}{h^2}\sin^2(\frac12\pi h)=\pi^2+O(h^2)$
\cite[\S{}II.3.2]{Sam}.
Combining this with \eqref{fdm} and \eqref{w_fdm} yields
\begin{align*}
\LL^h[W_{ij}\sin(\pi z_k)]&=\bigl[-(\delta_x^2+\delta_y^2) W_{ij} +
\gamma_{ij}\lambda_h W_{ij}\bigr]\sin(\pi z_k)\\[2pt]
&=\underbrace{\gamma_{ij}F(x_i,y_j)\sin(\pi z_k)}_{{}=\LL^h U_{ijk}} + O(h^2),
\end{align*}
where we also used $ |W_{ij}|\le \pi^{-2}\|F\|_{\infty}$.
The desired result follows by an application of the discrete maximum principle for the operator $\LL^h$.
\end{proof}

\begin{lemma}[{\cite{Andreev89_proc}}]\label{lem_andreev}
Let
$w$ solve \eqref{w_prob}, and
$\widetilde W_{ij}$ satisfy $-(\delta_x^2+\delta_y^2)\widetilde W_{ij}=\gamma_{ij}\widetilde F(x_i,y_j)$, subject to $\widetilde W_{ij}=0$ at the boundary nodes.
Then there exists a constant $C_1$, independent of $h$ and $\widetilde F$, such that
\beq\label{iterm}\textstyle
w(\frac12,\frac12)-\widetilde W_{\frac N2,\frac N2}\ge C_1 h^2|\ln h|\,\widetilde F(\frac12,\frac12)-O(h^2).
\eeq
\end{lemma}

\begin{proof}
Recalling that
 $\widetilde F =F-\pi^2 w$, rewrite \eqref{w_prob} as
 $-(\pt_x^2+\pt_y^2)w=\widetilde F$.
Now $\widetilde W_{ij}$ may be considered a finite difference approximation of $w$, for which
\eqref{iterm}
is obtained in \cite{Andreev89_proc}.
Note that $C_1$  is independent of $h$ and $\widetilde F$  (as it is related to the discrete Green's function
for the operator $-(\delta_x^2+\delta_y^2)$; see also Remark~\ref{rem_new}).
\end{proof}

\begin{remark}[$C_1$ in \eqref{iterm} {\cite{Andreev89_proc}}]\label{rem_new}
It is noted in \cite{Andreev89_proc} that $\widetilde W_{ij}$ of Lemma~\ref{lem_andreev} allows the representation
$\widetilde W_{ij}=\mathring{W}_{ij}-\frac13 h^2G_{ij}\,\widetilde F(\frac12,\frac12)$,
where $-(\delta_x^2+\delta_y^2)\mathring{W}_{ij}=\widetilde F(x_i,y_j)$, subject to $\mathring{W}_{ij}=0$ at the boundary nodes,
while $G_{ij}$ is the discrete Green's function
for the operator $-(\delta_x^2+\delta_y^2)$ associated with the node $(\frac N2,\frac N2)$.
To be more precise, $-(\delta_x^2+\delta_y^2)G_{ij}$ equals $h^{-2}$ if $i=j=\frac N2$ and $0$ otherwise, subject to $G_{ij}=0$ at the boundary nodes.
Furthermore, there is a constant $C_1>0$ independent of $h$ (as well as 
of $\widetilde F$) such that $G_{\frac N2,\frac N2}\ge 3C_1 |\ln h|$.
(For the latter, Andreev \cite{Andreev89_proc} uses \cite[(16)]{laasonen58};
see also \cite[Lemma~6]{scott76} for a similar result in the finite element context.)
Finally, for 
$\mathring{W}_{ij}$,
one has a standard finite-difference error bound $|\mathring{W}_{ij}-w(x_i,y_j)|\lesssim h^2$.
The above observations yield \eqref{iterm}.
\end{remark}

\begin{remark}
It was also pointed out in \cite{Andreev89_proc} that $\widetilde W_{ij}=w_h(x_i,y_j)$, where $w_h$ is a linear finite element solution for
 $-(\pt_x^2+\pt_y^2)w=\widetilde F$ obtained using the two-dimensional triangulation of $(0,1)^2$ shown on Fig.\,\ref{fig1} (left).
\end{remark}

\begin{lemma}\label{lem3}
Under the conditions of Theorem~\ref{theo_main} on $F$, for the solutions of \eqref{w_prob} and \eqref{w_fdm}
with a sufficiently small $h$,
one has
$
\!\displaystyle\max_{i,j=0,\ldots,N}|w(x_i,y_j)-W_{ij}|
\ge C_0 h^2|\ln h|$, with  a positive constant $C_0$  that depends on $F$, but not on $h$.
\end{lemma}

\begin{proof}
Set $C_0:=4 \pi^{-2} C_1 \|F\|_\infty/{\cosh(\frac12\pi)}$ and
 $e_{ij}:=\widetilde W_{ij}-W_{ij}$, where $C_1$ and $\widetilde W_{ij}$ are from Lemma~\ref{lem_andreev}.
Note that  $-(\delta_x^2+\delta_y^2)e_{ij}=\gamma_{ij}\pi^2[ W_{ij}-w(x_i,y_j)]$
 (in view of $\widetilde F =F-\pi^2 w$).
 Also, for the auxiliary $B_{ij}:=\frac12\pi^2  C_0 h^2|\ln h|\bigl\{\frac14-(x_i-\frac12)^2\bigl\}$,
 note that $-(\delta_x^2+\delta_y^2)B_{ij}=\pi^2  C_0 h^2|\ln h|$.
 We now prove the desired bound by contradiction.
Assume that
$\max_{i,j}|w(x_i,y_j)-W_{ij}|< C_0 h^2|\ln h|$.
Then
 $-(\delta_x^2+\delta_y^2)[B_{ij}\pm e_{ij}]> 0$.
 Now, an application of
the discrete maximum principle 
yields
$B_{ij}\pm e_{ij}\ge 0$.
So $|e_{ij}|\le B_{ij}$,
so
$|\widetilde W_{ij}-W_{ij}|=|e_{ij}|\le \frac18 \pi^2  C_0 h^2|\ln h|$.
Combining this with~\eqref{iterm}, one concludes that
$$
\textstyle
w(\frac12,\frac12)-W_{\frac N2,\frac N2} \ge \underbrace{\textstyle\bigl\{C_1 \widetilde F(\frac12,\frac12) -\frac18 \pi^2  C_0\bigr\}}_{{}\ge \frac18 \pi^2  C_0}
h^2|\ln h|-O(h^2)
\ge C_0 h^2|\ln h|
$$
for   a sufficiently small $h$,
where we also used
$C_1 \widetilde F(\frac12,\frac12)\ge \frac14 \pi^2 C_0$ (in view of~\eqref{tilde_F_bound}).
The above contradicts our assumption that
$\max_{i,j}|w(x_i,y_j)-W_{ij}|< C_0 h^2|\ln h|$.
\end{proof}

\begin{proof}[Proof of Theorem~\ref{theo_main}.]
It now suffices to combine the findings of Lemmas~\ref{lem1},~\ref{lem2} and~\ref{lem3}.
In particular,
$u(x_i,y_j,\frac12)-U_{i,j,\frac N2}=w(x_i,y_j)-W_{i,j}+O(h^2)$, by Lemmas~\ref{lem1} and~\ref{lem2}. So, in view of Lemma~\ref{lem3},
one gets $\|u-u_h\|_{L_\infty(\Omega)}\ge C^* h^2|\ln h|$ with any fixed positive constant $C^*<C_0$.
\end{proof}

\end{document}